\documentclass[10pt,amstex]{article}

\usepackage{amsmath}
\usepackage{amscd}
\usepackage{fullpage}
\usepackage{graphicx}
\usepackage{subcaption}
\usepackage{setspace}
\usepackage{listings}
\usepackage{pgf,tikz}
\usetikzlibrary{shapes, arrows}
\usetikzlibrary{decorations.pathmorphing,patterns}
\usepackage{amsfonts}
\usepackage{amssymb}
\usepackage{tabularx}
\usepackage{enumerate}
\usepackage[title,titletoc,toc]{appendix}
\usepackage{amsthm}
\usepackage{subcaption}
\usepackage{color}
\DeclareMathAlphabet{\mathpzc}{OT1}{pzc}{m}{it}

\newcommand{\zero}{^{(0)}}
\newcommand{\one}{^{(1)}}

\newcommand{\e}{\epsilon}

\newcommand{\Complex}{\mathbb{C}}

\newcommand{\MBold}[1]{\mathbf{#1}}
\newcommand{\MZero}{\MBold{0}}
\newcommand{\MId}{\MBold{I}}
\newcommand{\ML}{L}
\newcommand{\MP}{P}
\newcommand{\MG}{G}

\newtheorem{thmstyle}{Theorem}[section]

\newtheorem{Theorem}[thmstyle]{Theorem}

\newtheorem{Corollary}[thmstyle]{Corollary}
\newtheorem{Proposition}[thmstyle]{Proposition}

\newtheorem{Remark}[thmstyle]{Remark}

\DeclareMathOperator{\sgn}{sgn}

\DeclareMathOperator{\Avg}{Avg}

\newcommand{\RText}[1]{\text{#1}}
\newcommand{\Vect}[1]{\mathbf{#1}}

\newcommand{\ERef}[1]{\eqref{#1}}

\begin{document}

\title{Linear Nearest Neighbor Flocks with All Distinct Agents}
\author{R. Lyons
\thanks{Digimarc Corporation, 9405 SW Gemini Drive, Beaverton, OR, USA 97008-7192; Fariborz Maseeh Dept. of Math. and Stat., Portland State Univ.;
e-mail: rlyons@pdx.edu},
J. J. P. Veerman\thanks{Fariborz Maseeh Dept. of Math. and Stat., Portland State Univ.;
e-mail: veerman@pdx.edu}
}
\maketitle

\vskip .0in

\noindent

\begin{abstract}
This paper analyzes the global dynamics of 1-dimensional agent arrays 
with nearest neighbor linear couplings.   
The equations of motion are second order linear ODE’s with constant coefficients. 
The novel part of this research is that the couplings are different 
for each distinct agent.  
We allow the forces to depend on the positions and velocity (damping terms)
but the magnitudes of both the position and velocity couplings are different for each agent.
We, also, do not assume that the forces are ``Newtonian'' 
(i.e. the force due to A on B equals the minus the force of B on A) 
as this assumption does not apply to certain situations, such as traffic modeling. 
For example, driver A reacting to driver B does not imply the opposite reaction in driver B.

There are no known analytical means to solve these systems, 
even though they are linear, and so relatively little is known about them.
This paper is a generalization of previous work that computed
the global dynamics of 1-dimensional sequences of identical agents \cite{VeermanTransients_2016}
assuming periodic boundary conditions.
In this paper, we push that method further, similar to \cite{VeermanBaldivieso_2019},
and use an extended periodic boundary condition to 
to gain quantitative insights to the systems under consideration. 
We find that we can approximate the global dynamics of such a system 
by carefully analyzing the low-frequency behavior of the system 
with (generalized) periodic boundary conditions.
\end{abstract}

%
%
\section{Introduction}

A one-dimensional lattice of coupled agents is a model for many physical systems.
If all the agents are identical then connecting nearest neighbors with a Hooke’s Law
is a simple model of one-dimensional crystals \cite{Ashcroft_SolidStat_1976}.
If one assumes periodic boundary conditions, then the eigenvectors of the system
are the Discrete Fourier Transform basis functions.
In the 1950's this nearest neighbor crystal model was extended
to include agents of different mass \cite{Dyson_LinChain_PhysRev}.

In the 1950's simplified ``microscopic'' traffic models
appeared with agents coupled with a force dependent on spatial differences
and an added force that is a function of the difference of agent velocities \cite{TrafficCarFollowing_Montroll_1957}
(see \cite{KernerTraffic_2004} for a survey of traffic models).
The velocity dependent force term originated as an empirical law,
observed when an automobile attempts to follow a leader.
Since then the subject of cooperative control
has matured considerably \cite{Mackenroth_RobustControl_2004} \cite{RenBeard_ConcControl_2008}.

Recent technological advances make automated traffic platoons possible
so there has been renewed interest in one-dimensional lattice dynamics.
There are several works on both single
\cite{DistribControlAutoVeh_LiZheng_2017} \cite{LinFardadJovan2012VehicleFormNN}
and double integrator systems \cite{StringInstability_Middleton}
\cite{StabLargePlatoon_HaoBarooah_2013} \cite{DistribControlAutoVeh_LiZheng_2017}.
The results for both single and double integrator systems with nearest neighbor interactions
are summarized in \cite{LinFardadJovan2012VehicleFormNN}.

In the one-dimensional traffic platoon, one would like to know whether
is it possible, or even reasonable, to have a long platoon consisting of $N$ agents?
If we form a caravan of trucks,
do we need to break the caravan into separate small chunks
or can we form a single caravan of, perhaps, over a thousand trucks.
There is a large literature on the topic, but almost all the literature
we are aware of addresses this question in the unrealistic case
where all cars are identical or distributed in some other highly improbable way.
Each agent is distinct and may have a unique mass.
We can force the forward and backward couplings to have a specific ratio
but it is difficult and certainly impracticable to insist
that the force magnitudes are identical for all agents.
In section \ref{sec!symlap} we assign specific ratios for forward and backward couplings
as in equations \ERef{symlap!assume_rhox_equal} and \ERef{symlap!assume_rhov_equal},
but the weights $g_{x}^{(\alpha)}, g_{v}^{(\alpha)}$ are chosen randomly from a distribution.

More specifically, in this paper we shall analyze a one-dimensional lattice of agents
with linear nearest neighbor couplings
determined by the distance and the velocity difference between neighbors.
We will assume equations of motion where the force on an agent
is linear in position and velocity differences (double integrator system).
We shall not assume that the forces are ``Newtonian''
(i.e. the force due to A on B equals the minus the force of B on A).

In previous work (\cite{VeermanTransients_2016}, \cite{VeermanSigVel_2016}) it was shown that
if the agents are identical and the system has periodic boundary conditions
then the equations of motion are solvable.
This system has equations of motion given by the ODE,
\begin{equation}
\label{intro!eq_equofmotion}
 \frac{d}{dt} \begin{pmatrix}
   Z        \\
  \dot{Z}
 \end{pmatrix} = \begin{pmatrix}
     \MZero     &    \MId         \\
     \ML_{x}    &    \ML_{v}
 \end{pmatrix} \begin{pmatrix}
   Z        \\
  \dot{Z}
 \end{pmatrix} \; ,
\end{equation}
where $Z$ is a vector positions and $\ML_{x}, \ML_{v}$ are row-sum zero circulant tri-diagonal matrices.
Since the matrices are circulant, $\ML_{x}$ and $\ML_{v}$ commute.
This is instrumental in the finding solutions and the conditions of stability.
If the forces are extended to include next-nearest neighbor terms
then, assuming periodic boundary conditions, the equations are motion are,
also, given by equation \ERef{intro!eq_equofmotion}.
As in the nearest neighbor case, $\ML_x$ and $\ML_v$ are row-sum zero circulant matrices but, this time,
they have 5 non-zero diagonals.
The solutions are more complicated as are the conditions of stability, \cite{Herbrych2015DynamicsOL}.
For both systems the matrices $\ML_{x}$ and $\ML_{v}$ commute.
In both these cases the characteristic polynomial has a double root at $0$,
which corresponds to the stable configuration where all agents are moving at a constant velocity.
To find the asymptotic behavior of the system you can expand the zero locus around this point.
On stable systems, roots near the origin dominate the long-term behavior of the system
as other roots have larger negative real components and so decay faster.
Expanding the characteristic equation near the origin
yields an approximation to the signal velocity and a dispersion term.

However, in this work we do not assume the agents are identical.
Instead we introduce a repeating sequence of $p$ distinct agents.
Duplicate this string of $p$ agents and use an extension of periodic boundary conditions
which was first described in \cite{VeermanBaldivieso_2019}.
In particular, let $A_{0}, \cdots, A_{p-1}$ be $p$ agent types organized in a one-dimensional lattice,
\[
 A_{p-1} \leftrightarrow A_{p-2} \leftrightarrow
    \cdots
 \leftrightarrow A_{1} \leftrightarrow A_{0}.
\]
Then repeat this $p-$sequence $q$ times to get a total of $N = pq$ agents.
The general form for this system the matrices $\ML_{x}$ and $\ML_{v}$ {\bfseries do not commute}.
The case for $p = 2$ and $p = 3$ is analyzed in \cite{VeermanBaldivieso_2019}
for both nearest neighbor and next nearest neighbor interactions.
Since $\ML_{x}$ and $\ML_{v}$ do not commute the system is considerably more difficult to analyze,
but some conditions necessary for stability are derived.

In this paper we present a variety of tools to analyze this general system.
The goal is to first understand the periodic case
and then use these results to shed light on the general system of $N$ agents traveling on the real line.
In Section \ref{sec!nnsystem} Theorem \ref{nnsystem!thm_nec_for_stab}
we prove a generalization of the stability condition in \cite{VeermanBaldivieso_2019}.
To pursue the dynamics of a general system we start
with the system in equation \ERef{intro!eq_equofmotion},
where $\ML_{x}$ and $\ML_{v}$ are the circulant matrices with $-1$ on the diagonal
and $1/2$ on the sub and super diagonals.
This is the system in \cite{VeermanSigVel_2016} with
\[
   \rho_{x,1} = \rho_{x,-1} = -1/2
   \RText{  and  }
   \rho_{v,1} = \rho_{v,-1} = -1/2.
\]
This system is stable, \cite{VeermanSigVel_2016}.
We extend this by scaling each row by a distinct value,
which is the same as taking distinct weights $g_{x}^{(\alpha)}$ and $g_{v}^{(\alpha)}$.
In this case equation \ERef{intro!eq_equofmotion} becomes,
\begin{equation}
\label{intro!eq_equofmotion_Gs}
 \frac{d}{dt} \begin{pmatrix}
   Z        \\
  \dot{Z}
 \end{pmatrix} = \begin{pmatrix}
     \MZero             &    \MId         \\
     \MG_{x} \ML_{x}    &    \MG_{v} \ML_{v}
 \end{pmatrix} \begin{pmatrix}
   Z        \\
  \dot{Z}
 \end{pmatrix} \; ,
\end{equation}
where $\MG_{x}$ and $\MG_{v}$ are diagonal matrices with positive real values.
Again, we note that $\MG_{x} \ML_{x}$ and $\MG_{v} \ML_{v}$ do not commute, so this system is more difficult to analyze.
In Section \ref{sec!symlap} we prove stability for a special case of this general problem.
In Section \ref{sec!charpolyzero} we analyze the dynamics by expanding
the characteristic polynomial root locus around the double root at $0$.
As in the problems above the asymptotic behavior of the system is given by roots near the double root at $0$.
Expansion around this point results in an expression of the signal velocity and a dispersion term,
given in Theorem \ERef{cpolyzero!thm_secorder}.

The expansion, used in Section \ref{sec!charpolyzero}, requires an extension of the periodic boundary condition
first found in \cite{VeermanTransients_2016} and \cite{VeermanBaldivieso_2019}.
We take the $p$ distinct agents and repeat them $q$ times.
This guarantees that the locus is well approximated by a continuous curve as $q$ gets large,
and this allows us to use a Taylor expansion.
In the simulations in section \ref{sec!simul} we will set $q = 1$
and show that the results apply well to the general case of $N = p$ distinct agents.

\section{Linear Nearest Neighbor Systems}
\label{sec!nnsystem}

In this section we describe the equations of motion for one-dimensional lattice
with linear equations and nearest neighbor interactions.
In Theorem \ref{nnsystem!thm_nec_for_stab} below,
we prove a necessary condition for stability.
Later, in Section \ref{sec!symlap} we restrict to the case where the Laplacians
are symmetric matrices that have rows scaled by independent positive weights.
In this case we can derive some properties of the systems dynamics.

We consider $N$ agents consisting of $q$ groups of $p$ agents.
The first cluster of $p$ are unique and they are followed by $q$ identical clusters of $p$
so that the total number of agents is $N = pq$.
We follow the conventions in \cite{VeermanBaldivieso_2019},
except we change the signs so that $g_{x}, g_{v} \ge 0$.
After adjusting for the spacing, by subtracting $k \Delta$ from the $k$'th agent (see \cite{VeermanSigVel_2016}),
we can write the nearest neighbor coupling as,
\begin{equation}
\label{nnsystem!eq_ode_system}
\frac{ d^2 z_{k}^{(\alpha)} }{d t^2 }
  = - g_{x}^{(\alpha)} \left(
       z_{k}^{(\alpha)} + \rho_{x,1}^{(\alpha)} z_{k+1}^{(\alpha+1)}
              + \rho_{x,-1}^{(\alpha)} z_{k-1}^{(\alpha-1)}
    \right)
    - g_{v}^{(\alpha)} \left(
       \dot{z}_{k}^{(\alpha)} + \rho_{v,1}^{(\alpha)} \dot{z}_{k+1}^{(\alpha+1)}
              + \rho_{v,-1}^{(\alpha)} \dot{z}_{k-1}^{(\alpha-1)}
    \right)
\end{equation}
where the arithmetic in $\alpha$ is $\mod(p)$.
The coefficients satisfy $\rho_{x,-1}^{(\alpha)} + 1 + \rho_{x,1}^{(\alpha)} = 0$
and $\rho_{v,-1}^{(\alpha)} + 1 + \rho_{v,1}^{(\alpha)} = 0$
for all $\alpha = 0, \cdots, p-1$.

At this point we assume periodic boundary conditions.
and re-arrange vector components so the first $q$ coordinates are agents of type $\alpha = 0$,
the next $q$ agents are type $\alpha = 1$, etc.
This is a generalization of technique in \cite{VeermanBaldivieso_2019}.

To write the system in matrix form we start by writing an $N \times N$ matrix,
\begin{equation}
\ML_{\eta} = \begin{pmatrix}
    -\MId                & - \rho_{\eta,1}^{(0)} \MId         & \MZero
             & \cdots         &  -\rho_{\eta,-1}^{(0)} \MP_{-}
    \\
-\rho_{\eta,-1}^{(1)} \MId  & -\MId                           &   -\rho_{\eta,1}^{(1)} \MId
             &  \cdots        & \MZero
    \\
 \MZero                  & -\rho_{\eta,-1}^{(2)} \MId         &  -\MId
             &  \cdots        & \MZero
    \\
 \vdots                       & \vdots                        & \vdots
             &  \ddots        & \vdots
    \\
-\rho_{\eta,-1}^{(p-1)} \MP_{+}  &  \MZero             &  \MZero
             & \cdots         & -\MId
\end{pmatrix}
  \RText{  for } \eta = x, v,
\end{equation}
where $\MId$ and $\MZero$ are $q \times q$ matrices
and the matrix $\MP_+$ and its inverse $\MP_-$
are $q \times q$ cyclic permutations matrices,
\begin{align}
\MP_{+} =
\begin{pmatrix}
0 		&	1		&	0	&	\cdots		&	0         \\
0		&	0		&	1		&	\ddots		&	\vdots \\
\vdots	&	\ddots	&	\ddots		&	\ddots		&	0   \\
0		&	\ddots	&	\ddots		&		0		&  1   \\
1		&	0		&	\cdots		&	0		&	0	
\end{pmatrix}~,~~~
\MP_{-} =
\begin{pmatrix}
0 		&	0		&	\cdots	&	0		&	1\\
1		&	0		&	0		&	\ddots		&	0\\
0		&	1		&	\ddots		&	\ddots		&	\vdots \\
\vdots	&	\ddots	&		\ddots		&		0		&  0\\
0		&	\cdots	&		0		&	1			&	0	
\end{pmatrix}.
\label{permut_matrices}
\end{align}

We write the scales $g_{x}^{(\alpha)}, g_{v}^{(\alpha)}$ as an $N \times N$ matrix,
\begin{equation}
\MG_{\eta} = \begin{pmatrix}
   g_{\eta}^{(0)} \MId     &     \MZero     &    \MZero    &   \cdots     &    \MZero
    \\
   \MZero           &  g_{\eta}^{(1)} \MId  &    \MZero    &    \cdots    &    \MZero
    \\
   \MZero           &   \MZero       &  g_{\eta}^{(2)} \MId  &  \cdots    &    \MZero
    \\
  \vdots            &  \vdots        &   \vdots      &  \ddots     &    \vdots
    \\
  \MZero            &  \MZero        &   \MZero       &  \cdots    &  g_{\eta}^{(p-1)} \MId
\end{pmatrix}
  \RText{  for } \eta = x, v.
\end{equation}

The equations are motion are,
\begin{equation}
\begin{pmatrix}
\mathbf{\dot z}\zero \\
 \mathbf{\dot z}\one \\
 \vdots   \\
 \mathbf{\dot z}^{(p-1)}\\
 {\mathbf{\ddot z}}\zero \\
 {\mathbf{\ddot z}}\one \\
 \vdots   \\
 {\mathbf{\ddot z}}^{(p-1)}
 \end{pmatrix} =
 \left( \begin{array}{ccc|ccc}
       &                  &   &     &                   &
    \\
       &   \MZero         &   &     &   \MId            &
    \\
       &                  &   &     &                   &
    \\ \hline  
       &                  &   &     &                   &
    \\
       & \MG_{x} \ML_{x}  &   &     &  \MG_{v} \ML_{v}  &
    \\
       &                  &   &     &                   &
\end{array} \right)
\begin{pmatrix}
\mathbf{z}\zero \\
\mathbf{z}\one \\
 \vdots \\
\mathbf{z}^{(p-1)}      \\
\dot{\mathbf{z}}\zero \\
\dot{\mathbf{z}}\one  \\
 \vdots \\
\dot{\mathbf{z}}^{(p-1)}
\end{pmatrix} \;.
\label{eq:linear_sys3}
\end{equation}

The matrices $\Vect{P}_{+}$ and $\Vect{P}_{-}$ are $q \times q$ circulant matrices
and all circulant matrices have a common set of $q$ eigenvectors given by,
\[
\Vect{v}_{m}
  = \left[ 1, \exp\left( \frac{2 \pi i }{ q } m \right),
              \exp\left( \frac{2 \pi i }{ q } 2 m \right),
              \cdots,
              \exp\left( \frac{2 \pi i }{ q } (q-1) m \right)
    \right]^{T}
\]
where $m = 0, 1, \cdots, q-1$.

To simplify the notation, we introduce the following 1st degree polynomials,
\begin{align}
\label{agentseq!def_psi_00}
  \psi_{0}^{(\alpha)}(\nu) &= \left( g_{v}^{(\alpha)} \nu + g_{x}^{(\alpha)} \right)
    &
  \psi_{1}^{(\alpha)}(\nu) &= g_{v}^{(\alpha)} \rho_{v,1}^{(\alpha)} \nu + g_{x}^{(\alpha)} \rho_{x,1}^{(\alpha)}
   &
  \psi_{-1}^{(\alpha)}(\nu) &= g_{v}^{(\alpha)} \rho_{v,-1}^{(\alpha)} \nu + g_{x}^{(\alpha)} \rho_{x,-1}^{(\alpha)}
\end{align}

For all $\alpha, \nu$, we have,
\begin{equation}
\label{rlproofs!sum_psis_vanish}
\psi_{-1}^{(\alpha)}(\nu)
 + \psi_{0}^{(\alpha)}(\nu)
 + \psi_{1}^{(\alpha)}(\nu)
  = 0.
\end{equation}

The eigenvalue equation of the general system is simplified by using the following,

\begin{Proposition}
The eigenvectors of the matrix in equation \ERef{linnn!eq_eigenvect} are given by the $2N= 2 p q$ vectors,
\begin{equation}
\label{linnn!eq_eigenvect}
\Vect{u}_{\nu}( m, \phi )
  = \begin{pmatrix}
    \e_{0} \Vect{v}_{m}  \\ \e_{1} \Vect{v}_{m} \\ \vdots  \\
            \e_{p-1} \Vect{v}_{m}
       \\
    \nu  \e_{0} \Vect{v}_{m}  \\ \nu \e_{1} \Vect{v}_{m} \\ \vdots  \\
             \nu \e_{p-1} \Vect{v}_{m}
  \end{pmatrix}
  = \begin{pmatrix}
     \Vect{v}_{m} \otimes \Vect{\epsilon}      \\
     \nu \left( \Vect{v}_{m} \otimes \Vect{\epsilon} \right)
  \end{pmatrix}
  \;.
\end{equation}
which have eigenvalue $\nu$ and where
\[
\Vect{\epsilon} = \begin{pmatrix}
      \epsilon_{0}  &   \epsilon_{2}  &   \cdots  &   \epsilon_{p-1}
   \end{pmatrix}^{T}.
\]
For each $m \in \{0,\cdots q-1\}$ there are $2p$ eigenvalues
given by the determinant of the matrix in equation \ERef{eq:mat_epsilon_nullity}.
For each of these roots the eigenvector uses the values $\epsilon_{j}$ that satisfy,
\begin{equation}
\label{eq:mat_epsilon_nullity}
\begin{pmatrix}
   \nu^2 + \psi_{0}^{(0)}( \nu)  &  \psi_{1}^{(0)}( \nu)               &       0
                    &   \cdots     &  \psi_{-1}^{(0)}( \nu) e^{- i \phi }
     \\
  \psi_{-1}^{(1)}( \nu)             &   \nu^2 + \psi_{0}^{(1)}( \nu )  &  \psi_{1}^{(1)}( \nu )
                   &   \cdots      &    0
    \\
  \vdots                                &         \vdots                            &   \vdots
                   &  \ddots      &   \vdots
    \\
  \psi_{1}^{(p-1)}( \nu) e^{ i \phi }    &      0                             &      0
                    &    \cdots   &   \nu^2 + \psi_{0}^{(p-1)}( \nu)
\end{pmatrix}
\begin{pmatrix}
\e_0  \\
\e_1  \\
\vdots \\
\e_{p-1}
\end{pmatrix}
=
\begin{pmatrix}
0\\ 0\\ \vdots \\ 0
\end{pmatrix}\;.
\end{equation}
\label{prop:charpoly}
\end{Proposition}

\begin{proof}
The proof of this proposition is a generalization
of the proof of proposition 2.1 from \cite{VeermanBaldivieso_2019}.
Apply the matrix in equation \ERef{eq:linear_sys3}
to the vector in equation \ERef{linnn!eq_eigenvect}.
Use,
\[
\begin{array}{cc}
 \mathbf{P}_{-}  \Vect{v}_{m} = e^{- i \phi } \Vect{v}_{m}
 \;,
  &
 \mathbf{P}_{+}  \Vect{v}_{m} = e^{i \phi } \Vect{v}_{m}.
\end{array}
\]
where we define $\phi = \frac{ 2 \pi }{q} m$.
The top $N$ coordinates follow immediately
and the bottom $N$ coordinates yield $q$ copies of the equation \ERef{eq:mat_epsilon_nullity}.
\end{proof}

\begin{Corollary}
\label{nnsystem!cor_eigsofpoly}
The eigenvalues of the system are the roots to the polynomial,
\begin{equation}
\label{nnsystem!cor_eigs_poly}
P_{\phi}(\nu)
 = \det \begin{pmatrix}
   \nu^2 + \psi_{0}^{(0)}( \nu)  &  \psi_{1}^{(0)}( \nu)               &       0
                    &   \cdots     &  \psi_{-1}^{(0)}( \nu) e^{- i \phi }
     \\
  \psi_{-1}^{(1)}( \nu)             &   \nu^2 + \psi_{0}^{(1)}( \nu )  &  \psi_{1}^{(1)}( \nu )
                   &   \cdots      &    0
    \\
  \vdots                                &         \vdots                            &   \vdots
                   &  \ddots      &   \vdots
    \\
  \psi_{1}^{(p-1)}( \nu) e^{ i \phi }    &      0                             &      0
                    &    \cdots   &   \nu^2 + \psi_{0}^{(p-1)}( \nu)
\end{pmatrix}.
\end{equation}
For each value of $\phi = \frac{2 \pi i}{q} m$, $m = 0, \cdots q-1$,
there are $2p$ roots.
\end{Corollary}

\begin{Proposition}
\label{rlproofs!poly_phi0_quad}
When $\phi = 0$ (e.g. $m=0$) the constant and linear terms for the polynomial $P_{0}( \nu )$ both vanish.
\end{Proposition}
\begin{proof}
Neither the linear nor constant terms of the polynomial cannot have $\nu^2$ as a factor.
Set $\phi = 0$ and remove the terms with $\nu^2$ as a factor.
The resulting polynomial is,
\[
\det \begin{pmatrix}
   \psi_{0}^{(0)}( \nu)  &  \psi_{1}^{(0)}( \nu)               &       0
                    &   \cdots     &  \psi_{-1}^{(0)}( \nu)
     \\
  \psi_{-1}^{(1)}( \nu)             &   \psi_{0}^{(1)}( \nu )  &  \psi_{1}^{(1)}( \nu )
                   &   \cdots      &    0
    \\
  \vdots                                &         \vdots                            &   \vdots
                   &  \ddots      &   \vdots
    \\
  \psi_{1}^{(p-1)}( \nu)      &      0                             &      0
                    &    \cdots   &   \psi_{0}^{(p-1)}( \nu)
\end{pmatrix}
\]
Every row in this matrix sums to zero, by equation \ERef{rlproofs!sum_psis_vanish}.
This means the vector, consisting of all $1$'s is an eigenvector with eigenvalue $0$
and so the determinant vanishes for all $\nu$.
The constant and linear terms are contained in this reduced polynomial and so must vanish.
\end{proof}

\begin{Remark}
\label{rlproofs!rem_phizero_poly}
The proof of Proposition \ref{rlproofs!poly_phi0_quad}
says the second order term of the polynomial $P_{0}( \nu )$
must contain exactly one of the diagonal $\nu^2$ terms.
The third order term of $P_{0}(\nu)$ must, also, contain exactly one of $\nu^2$ terms.
These facts will be used below.
\end{Remark}

\begin{Proposition}
\label{rlproofs!prop_phi_depends}
\[
P_{\phi}( \nu )
  = s( \nu ) + (-1)^{p} r_{\phi}( \nu ),
\]
where all the $\phi$ dependence is in the polynomial,
\[
r_{\phi}( \nu )  =
      (1 -  e^{ i \phi } )   \left( \psi_{1}^{(0)}( \nu) \psi_{1}^{(1)}( \nu)
           \cdots \psi_{1}^{(p-1)}( \nu)  \right)
    + (1 - e^{ - i \phi } ) \left(  \psi_{-1}^{(0)}( \nu) \psi_{-1}^{(1)}( \nu)
           \cdots \psi_{-1}^{(p-1)}( \nu)   \right)
\]
and $s(\nu)$ has zero constant and linear terms.
\end{Proposition}
\begin{proof}
The only terms of the expansion of equation \ERef{nnsystem!cor_eigs_poly} that depend on $\phi$
contain either $\psi_{-1}^{(0)}( \nu ) e^{-i\phi}$
or $\psi_{1}^{(p-1)}( \nu ) e^{i\phi}$ but not both.
The determinant is a sum over permutations $\sigma$ of terms $\sgn(\sigma) M_{0 \sigma 0} \cdots M_{p-1 \sigma p-1}$.
The only non-zero permutations have $\sigma(k) \in \{ k-1, k, k+1 \} MOD(p)$.
The permutations that contain the term $\psi_{-1}^{(0)}( \nu ) e^{-i\phi}$
but not $\psi_{1}^{(p-1)}( \nu ) e^{i\phi}$ must have $\sigma(0) = p-1$.
The matrix is tri-diagonal so the value $\sigma(p-1) \in \{ 0, p-1, p-2 \}$ for all $\sigma$.
But, in this case we know $\sigma(p-1) \ne p-1$ and,
by assumption, $\sigma(p-1) \ne 0$ (or the term would contain the $e^{i\phi}$ term).
Therefore, $\sigma(p-1) = p-2$.
By a similar logic, $\sigma(p-2) \in \{ p-3, p-2, p-1 \}$
but $\sigma(p-2) \ne p-2$ and $\sigma(p-2) \ne p-1$.
So we get $\sigma(p-2) = p-3$.
Proceed in this way to get the permutation, $\sigma = \left( 0, p-1, p-2, \cdots 1 \right)$
which has $\sgn( \sigma ) = (-1)^{p-1}$.
This corresponds to the term,
\[
- (-1)^{p} e^{ - i \phi }
  \psi_{-1}^{(0)}( \nu) \cdots \psi_{-1}^{(p-1)}( \nu)
\]
The term that contains $\psi_{1}^{(p-1)}( \nu ) e^{i\phi}$
but not $\psi_{-1}^{(0)}( \nu ) e^{-i\phi}$ is computed in a similar way,
and seen to be,
\[
- (-1)^{p} e^{ i \phi }
  \psi_{1}^{(0)}( \nu) \cdots \psi_{1}^{(p-1)}( \nu)
\]

We define $r_{\phi}$ so that $r_{0}(\nu) = 0$.
\newline

From Proposition \ref{rlproofs!poly_phi0_quad}
it follows immediately that the constant and linear terms of $s(\nu)$ both vanish.
\end{proof}

A consequence of this Proposition is the following.
\begin{Corollary}
\label{agentseq!cor_deriv_pphi}
\[
\left. \frac{ d^{k} P_{\phi} }{ d \phi^{k}  } \right|_{\phi = 0}
  = (-1)^{p+1} i^{k}     \left( \psi_{1}^{(0)}( \nu) \psi_{1}^{(1)}( \nu)
           \cdots \psi_{1}^{(p-1)}( \nu)  \right)
    + (-1)^{p+1} (-i)^{k} \left(  \psi_{-1}^{(0)}( \nu) \psi_{-1}^{(1)}( \nu)
           \cdots \psi_{-1}^{(p-1)}( \nu)   \right)
\]
\end{Corollary}

From this we can prove a necessary condition for stability.
\begin{Theorem}
\label{nnsystem!thm_nec_for_stab}
If, for a general (linear) nearest neighbor system,
\begin{align*}
\prod_i\;\rho_{x,1}^{(i)} -  \prod_i\;\rho_{x,-1}^{(i)} \neq 0
\end{align*}
the system is unstable in one sense or another.
\label{thm:general_neccond}
\end{Theorem}
\begin{proof}
By \cite{VeermanBaldivieso_2019} (specifically, see Appendix in \cite{VeermanBaldivieso_2019}),
the constant term of $\left. \frac{ d P_{\phi} }{ d \phi  } \right|_{\phi = 0}$ must vanish.
By Corollary \ref{agentseq!cor_deriv_pphi} the derivative of the constant term is,
\[
 (-1)^{p+1} i  \left(
       \rho_{x,1}^{(0)} \rho_{x,1}^{(1)}( \nu)  \cdots \rho_{x,1}^{(p-1)}
   -   \rho_{x,-1}^{(0)} \rho_{x,-1}^{(1)} \cdots \rho_{x,-1}^{(p-1)}
    \right)
\]
The $\rho_{x,1}^{(\alpha)}, \rho_{x,-1}^{(\alpha)}$ are all real so the theorem follows.
\end{proof}

In this general case it is difficult
to come up with sufficient conditions for stability.
If we simplify the problem a bit there is more that can be said.

\section{$L_x$ and $L_{v}$ are Symmetric Laplacians}
\label{sec!symlap}

In \cite{VeermanSigVel_2016} we showed that when $p = 1$
then stable systems must have $\rho_{x,1} = \rho_{x,-1} = -1/2$.
Theorem \ref{thm:general_neccond} indicates that this is a reasonable assumption
and so, henceforth, we shall assume that,
\begin{equation}
\label{symlap!assume_rhox_equal}
  \rho_{x,1}^{(\alpha)} = \rho_{x,-1}^{(\alpha)} = - \frac{1}{2}
       \RText{ for all }
       \alpha = 0, \cdots p-1.
\end{equation}
To make the problem tractable we shall also assume that,
\begin{equation}
\label{symlap!assume_rhov_equal}
  \rho_{v,1}^{(\alpha)} = \rho_{v,-1}^{(\alpha)} = - \frac{1}{2},
       \RText{ for all }
       \alpha = 0, \cdots p-1.
\end{equation}

With these two assumptions $L_x$ and $L_v$ are symmetric row-sum zero matrices.
With these assumptions we, also, have,
\begin{equation}
\label{symlap!assume_psi_equal}
 \psi_{1}^{(\alpha)}( \nu ) = \psi_{-1}^{(\alpha)}( \nu )
       \RText{ for all }
       \alpha = 0, \cdots p-1
       \RText{ and for all }
       \nu.
\end{equation}

One nice feature of the symmetric case is that one can prove stability in a restrictive sense
using the following fact.

\begin{Proposition}
\label{symlap!prop_Mstab_GMStab}
If $M$ is a diagonalizable matrix with eigenvalues in the left half complex plane
and $G$ is a positive definite matrix
then $GM$ has eigenvalues in the left half complex plane.
\end{Proposition}
\begin{proof}
Variants of this Proposition are known.
We include the proof for completeness.
If $G$ is a positive definite matrix then there is a non-singular square root $G^{1/2}$.
For any $x$ there is a $y$ with $x = G^{1/2} y$.
So, for any $x$ we have,
\begin{align*}
\langle { G M y, y } \rangle
  &= \langle { G^{1/2} G^{1/2} M y, y } \rangle
   = \langle { G^{1/2} M y, G^{1/2} y } \rangle
   = \langle { G^{1/2} M G^{-1/2} G^{1/2} y, G^{1/2} y } \rangle
    \\
  &= \langle { G^{1/2} M G^{-1/2} x, x } \rangle.
\end{align*}
The eigenvalues of $G^{1/2} M G^{-1/2}$ are the same as $M$.
So, for any $y$ we have,
\[
\Re \left( \langle { G M y, y } \rangle \right)
  \le 0.
\]
\end{proof}

\begin{Theorem}
\label{symlap!thm_LvLx_stable}
Let $L = L_x = L_{v}$ be a diagonalizable Laplacian
with eigenvalues in the closed left half plane.
Let $G_x = G_v = G$ be a diagonal positive matrix and let $\alpha > 0$ be a real number.
Then the roots of the characteristic polynomial of,
\begin{equation}
 \begin{pmatrix}
    \MZero       &    \MId
    \\
   G L_x        &  \alpha G L_v
 \end{pmatrix},
\end{equation}
all lie in the closed left half plane.
There is a double root at $0$ and all other roots have real part that is strictly negative.
\end{Theorem}
\begin{proof}
In \cite{VeermanSigVel_2016} the stability conditions are derived for Linear systems with a matrix,
\begin{equation}
\label{symlap!eq_LxLv_GisI}
 \begin{pmatrix}
    \MZero      &    \MId
    \\
    L_x        &    \alpha L_v
 \end{pmatrix}.
\end{equation}
Because $L_x$ and $\alpha L_v$ are commuting symmetric matrices
there is a complete set of eigenvectors with eigenvalues $\lambda_{x,m}$ and $\lambda_{v,m}$
with $\Re( \lambda_{x,m} ) \le 0$ and $\Re( \lambda_{v,m} ) \le 0$.
The eigenvalues of the matrix in equation \ERef{symlap!eq_LxLv_GisI} are given by,
\[
\nu_{m \pm}
  = \frac{ \lambda_{v,m}}{2} \pm \sqrt{ \frac{ \lambda_{v,m}^2}{4} + \lambda_{x,m} }
\]
These roots are all in the closed left half plane.

If $G$ is a diagonal positive matrix then
\[
 \begin{pmatrix}
    \MZero     &    \MId
    \\
    G L_x      &     G L_v
 \end{pmatrix}
  =  \begin{pmatrix}
    \MId       &    \MZero
    \\
   \MZero      &     G
 \end{pmatrix}  \begin{pmatrix}
    \MZero      &    \MId
    \\
     L_x        &     L_v
 \end{pmatrix}
\]
Now use Proposition \ref{symlap!prop_Mstab_GMStab}.
\end{proof}

\begin{Remark}
With the assumptions in Theorem \ref{symlap!thm_LvLx_stable} we have,
\[
 \left[ G L_x, \alpha G L_v \right] = 0.
\]
\end{Remark}

\section{Characteristic Polynomial Expansion}
\label{charpolyexpand!start}

The roots of the characteristic polynomial of the general system, described in Section \ref{sec!nnsystem},
are given by the roots of a series of $p$ degree polynomials $P_{\phi}( \nu )$,
as described in Corollary \ref{nnsystem!cor_eigsofpoly}.
The characteristic polynomial $P_{\phi}( \nu )$ has a root of multiplicity $2$ at $z = 0$.
In this section we will
assume equations \ERef{symlap!assume_rhox_equal} and \ERef{symlap!assume_rhov_equal},
so that $\rho_{x,1}^{(\alpha)} = \rho_{x,-1}^{(\alpha)}$
and $\rho_{v,1}^{(\alpha)} = \rho_{v,-1}^{(\alpha)}$.
We will not assume $G_{x} = \alpha G_{y}$, as in Theorem \ref{symlap!thm_LvLx_stable}.
Instead we will let $g_{x}^{\alpha}$ and $g_{v}^{\alpha}$ be independent random variables.

If the system is stable, then roots of the characteristic polynomial
all have non-positive real parts.
Stable roots with large negative real part decay quickly
so a stable system is dominated by roots that lie near the imaginary axis.
In our system roots near the double zero will dominate the dynamics
so we will expand around this zero to approximate the system dynamics.
The details of the expansion are outlined in this section.

When $\phi = 0$ there is a double root of $P_{0}( \nu )$ at $\nu = 0$.
We would like to compute what happens to this double zero
when $\phi = \frac{ 2 \pi }{q} m$ is small but non-zero.
If $q$ is large enough we can approximate the system by using a continuous variable for $\phi$.
With this approximation, each of the roots at $\phi = 0, \nu = 0$
form continuous zero loci as $\phi$ varies.
We get two zero loci, which are functions,
\begin{equation}
\label{charpolyexpand!eq_gammaCurve}
 \gamma : I \to \Complex,
\end{equation}
where $I = (-\epsilon, +\epsilon)$ is some neighborhood of $0$.
These curves satisfy,
\begin{equation}
\label{charpolyexpand!eq_Pphigamma}
  P_{\phi}( \gamma( \phi )) = 0.
\end{equation}

The coefficients of the characteristic polynomial are analytic functions of $\phi$
so we will expand everything in Taylor series and use this equation to deduce conditions on the coefficients.
Assume that $\gamma(0) = 0$ and write the expansion,
\begin{equation}
\label{charpolyexpand!eq_gamma_taylor}
\gamma( \phi ) = \gamma'(0) \phi + \frac{1}{2} \gamma''(0) \phi^2 + \cdots.
\end{equation}
The coefficients of the polynomial $p_{\phi}( \nu )$ are real analytic functions of $\phi$.
This means we can expand each of them in a Taylor series.
The result is an expansion of the form,
\begin{equation}
\label{charpolyexpand!eq_Pphi_expand}
 P_{\phi}( \nu )
   = \left( a_{00} + a_{01} \phi + \cdots \right)
         + \left( a_{10} + a_{11} \phi + \cdots \right) \nu
         + \left( a_{20} + a_{21} \phi + \cdots \right) \nu^2
         + \cdots
\end{equation}
where the coefficients $a_{0k}, a_{1k}, \cdots,$ arise as the coefficients
of the $k$th derivative of $p_{\phi}( \nu )$ with respect to $\phi$,
\begin{equation}
\frac{1}{k!} \left. \frac{ d^{k} P_{\phi}}{ d \phi^{k} } \right|_{\phi = 0}
    = a_{0k} + a_{1k} \nu + a_{2k} \nu^2 + \cdots.
\end{equation}

\begin{Proposition}
\label{agentseq!prop_gamma_expand}
The expansion of the zero locus to second order gives the coefficients,
\begin{align}
\label{agentseq!taylor_solve_c1}
  \gamma'(0) &= \pm \sqrt{ - \frac{ a_{02} }{  a_{20} } }
       \\
\label{agentseq!taylor_solve_c2}
  \gamma''(0) &=
   - \frac{ a_{30} (\gamma'(0))^3 + a_{21} ( \gamma'(0))^2 + a_{12} \gamma'(0) + a_{03} }
          { a_{20} \gamma'(0) }
\end{align}
\end{Proposition}
\begin{proof}
The equation $P_{\phi}( \gamma( \phi )) = 0$ expands to a power series in $\phi$.
Set $\nu = \gamma( \phi )$ and expand using equation \ERef{charpolyexpand!eq_gamma_taylor}.
Plug this value of $\nu$ into the polynomial of equation \ERef{charpolyexpand!eq_Pphi_expand}.
Condition \ref{charpolyexpand!eq_Pphigamma} says this expansion in $\phi$ vanishes identically.
When you solve for the derivatives $\gamma^{(m)}(0)$ in terms of $a_{jk}$
you get equations \ERef{agentseq!taylor_solve_c1} and \ERef{agentseq!taylor_solve_c2}.
This is a straightforward calculation
and you can check the results using algebraic manipulation software, like SAGE.
\end{proof}

Notice that there are two solutions to the first order coefficient $c_1$.
There is a double root that splits into two distinct curves,
so there are two distinct values for $\gamma^{(1)}(0)$ and $\gamma^{(2)}(0)$.

At this point we find the coefficients $a_{jk}$ that are required for our expansion.
The reader may find it more digestible to jump to Section \ref{sec!charpolyzero}
and refer to the remaining portion of this section as needed.
We start with the following Proposition.

\begin{Proposition}
\label{agentseq!cor_deriv_akm}
If we assume equation \ERef{symlap!assume_rhox_equal} and \ERef{symlap!assume_rhov_equal} then we have,
\begin{align}
\label{agentseq!cor_eq_a_mk}
 a_{mk}  &= 0, \text{ for }k\text{ odd.}
    \\
\label{agentseq!cor_eq_a_02}
 a_{02} &= (-1)^{p} (\rho_{x,1})^{p} g_{x}^{(0)} \cdots g_{x}^{(p-1)}
          = 2^{-p} \prod_{j=0}^{p-1} { g_{x}^{(j)} }
    \\
\label{agentseq!cor_eq_a_12}
 a_{12} &= 2^{-p} \prod_{j=0}^{p-1} { g_{x}^{(j)} }
                 \sum\limits_{k=0}^{p-1} { \frac{ g_{v}^{(k)} }{ g_{x}^{(k)} } }
\end{align}
\end{Proposition}
\begin{proof}
Since  $\psi_{1}(\nu) = \psi_{-1}(\nu)$ then, by Corollary \ref{agentseq!cor_deriv_pphi}, we have
\[
\frac{1}{ k! } \left. \frac{ d^{k} P_{\phi} }{ d \phi^{k}  } \right|_{\phi = 0}
  = \frac{1}{ k! } (-1)^{p+1} ( i^{k} + (-i)^{k} )  \left( \psi_{1}^{(0)}( \nu) \psi_{1}^{(1)}( \nu)
           \cdots \psi_{1}^{(p-1)}( \nu)  \right).
\]
The coefficients are just the coefficients of this polynomial in $\nu$.
The coefficient $a_{mk}$ is the $m$'th polynomial coefficient of the polynomial of the $k$'th derivative.
The coefficients $a_{mk} = 0$ whenever $k$ is odd.
When $k = 2$ we have the polynomial,
\begin{align*}
\frac{1}{ 2 } \left. \frac{ d^{2} P_{\phi} }{ d \phi^{2}  } \right|_{\phi = 0}
  &= (-1)^{p} \left( \psi_{1}^{(0)}( \nu) \psi_{1}^{(1)}( \nu) \cdots \psi_{1}^{(p-1)}( \nu)  \right)
  \\
  &= a_{02} + a_{12} \nu + \cdots.
\end{align*}
To form the linear term in $\nu$ we can take the linear term in each $\psi_{1}^{(\alpha)}( \nu )$
where all the other terms contribute a constant term.   The result is that the linear term is
\[
a_{12}
 = (-1)^{p}  \sum\limits_{j=0}^{p-1} { g_{v}^{j} \rho_{v,1}^{j} \frac{ g_{x}^{0} \cdots g_{x}^{(p-1)} }{ g_{x}^{j} } \left( \rho_{x,1} \right)^{p-1} }
 = \frac{1}{2^{p}} \prod_{j=0}^{p-1} { g_{x}^{(j)} }
                 \sum\limits_{j=0}^{p-1} { \frac{ g_{v}^{(j)} }{ g_{x}^{(k)} } }.
\]
The term $a_{02}$ comes from the constant term which is computed in a similar way.
\end{proof}

\begin{Proposition}
\label{agentseq!prop_term_a20}
\label{agentseq!prop_term_a30}
Assume equation \ERef{symlap!assume_rhox_equal} and \ERef{symlap!assume_rhov_equal}.
The degree $2$ term of the characteristic polynomial, when $\phi = 0$, is given by
\begin{equation}
\label{agentseq!eqnu2_second}
a_{20}
  = \frac{p}{2^{p-1}} \prod_{j=0}^{p-1} { g_{x}^{(j)} } \left(
         \sum\limits_{k=0}^{p-1} { \frac{ 1 }{  g_{x}^{(k)} } }
    \right)
  = \frac{p^2}{2^{p-1}} \prod_{j=0}^{p-1} { g_{x}^{(j)} }  \Avg \left(  \frac{ 1 }{  g_{x}^{(k)} } \right)
\end{equation}

The degree $3$ term of the characteristic polynomial, when $\phi = 0$, has the value
\begin{align}
\label{agentseq!eqnu3_third}
a_{30}
  &= \frac{(2p)}{2^p}
 \left( g_{x}^{(0)} \cdots  g_{x}^{(p-1)} \right)
 \sum\limits_{k=0}^{p-1} \sum\limits_{\substack{j=0 \\ j \neq k}}^{p-1} {
    \frac{g_v^{(j)}}{g_x^{(k)} g_x^{(j)}}
  }
  \nonumber
  \\
  &= \frac{(2p)}{2^p}
  \prod_{j=0}^{p-1} { g_{x}^{(j)} }
  \left(
    \sum\limits_{k=0}^{p-1} {
      \frac{1}{g_x^{(k)}}
    }
    \sum\limits_{j=0}^{p-1} {
      \frac{g_v^{(j)}}{ g_x^{(j)}}
    }
   -
    \sum\limits_{k=0}^{p-1} {
      \frac{g_v^{(k)}}{ g_x^{(k)} g_x^{(k)} }
    }
  \right)
\end{align}
where the last sum is only over $(i,j)$ where $0 = i < j \le (p-1)$.
\end{Proposition}
\begin{proof}
To compute the second order term in $s( \nu )$ we see, from
Proposition \ref{rlproofs!poly_phi0_quad} that all $2$nd order terms
must contain a $\nu^2$ from the diagonal and all the remaining terms are constants.
Similarly, to compute the $3$rd order terms in $s( \nu )$,
by Proposition \ref{rlproofs!poly_phi0_quad}, all the cubic terms
that do not include a $\nu^2$ diagonal factor must sum to zero.
So the $3$rd order terms include a $\nu^2$ term from the diagonal
and a $\nu$ factor from one of the $\psi$'s.

The co-factor of the diagonal term have
the form assumed in Proposition \ref{prop_trimat_det}.
The $\nu^2$ in the $(k,k)$ position has a co-factor of the form,
\begin{align*}
 & (\sgn(\eta))^2 \det \left| \begin{matrix}
    \psi_0^{(k+1)}(\nu)  &   \psi_{1}^{(k+1)}(\nu)  &        0
           &   0   &   \cdots    &    0    &    0
    \\
    \psi_{-1}^{(k+2)}(\nu) & \psi_0^{(k+2)}(\nu)  &   \psi_{1}^{(k+2)}(\nu)
           &   0   &   \cdots    &    0    &    0
    \\
    \vdots      &    \vdots    &   \vdots
           &  \vdots  &  \cdots   &   \vdots    & \vdots
    \\
       0        &        0     &      0
           &   0     &    \cdots    &   \psi_0^{(k-1)}(\nu)  &   \psi_{1}^{(k-1)}(\nu)
  \end{matrix} \right]
  \\
 = & (-1)^{p-1} \left(
      \psi_{1}^{(k+1)}(\nu) \cdots \psi_{1}^{(k+p-1)}(\nu)
       + \cdots +
      \psi_{-1}^{(k+1)}(\nu) \cdots \psi_{-1}^{(k+p-1)}(\nu)
   \right)
\end{align*}
where all the coefficient arithmetic is assumed $\pmod{p}$
and $\eta \in S_{p}$ is the permutation that rotates the indices
so that $k$ is rotated to position $0$.
Each term in the last sum has $p-1$ factors.
The constant term is easily isolated and
all $p$ terms in the sum are equal to
\[
(-1)^{p-1} ( \rho_{x,1} )^{p-1} p \frac{ g_{x}^{(0)} g_{x}^{(1)} \cdots g_{x}^{(p-1)} }{ g_{x}^{(k)} }
 = \frac{p}{2^{p-1}} \frac{ g_{x}^{(0)} g_{x}^{(1)} \cdots g_{x}^{(p-1)} }{ g_{x}^{(k)} }.
\]
Equation \ERef{agentseq!eqnu2_second} follows.

For the cubic term, exactly one of the $\psi$ terms will contribute a $\nu$,
and so a $\rho_{v,\pm 1}$, which we label the $j$th term.
The contribution from the $k$'th $\nu^2$ diagonal
will have $p-2$ factors of $\rho_{x,\pm 1}$ and has the form,
\begin{align*}
(-1)^{p-1} (\rho_{x,1})^{p-2}
  & \sum\limits_{j=k+1}^{k-1+p} {
      \frac{ g_{x}^{(0)} \cdots g_{x}^{(p-1)} }{ g_{x}^{(k)} g_{x}^{(j)} } g_{v}^{(j)}
      \left(  \rho_{v,1} (p-1-j)  +   \rho_{v,-1} (j+1) \right)
  }
  \\
 &= \frac{ p }{ 2^{p-1}}
  \sum\limits_{\substack{j=0 \\ j \neq k}}^{p-1} {
    \frac{g_x^{(0)} \cdots g_{x}^{(p-1)}}{ g_{x}^{k} g_{x}^{j} }
     g_{v}^{j}
  }
\end{align*}
recall that all index arithmetic is $\mod(p)$.
The formula for $a_{30}$ follows.
\end{proof}

\begin{Proposition}
\label{prop_trimat_det}
Let $D_{n}$ be the determinant,
\begin{equation}
D_{n} = \left| \begin{matrix}
     ( c_{1} + d_{1} ) & -c_{1}  &   0    &     0       &   \cdots
        &      0      &     0
    \\
   -d_{2}  &   (d_{2} + c_{1} ) & -c_{2}  &   0    &   \cdots
        &      0      &     0
    \\
      0    &   -d_{3}  &   (d_{3} + c_{3} ) & -c_{3}  & \cdots
        &      0      &     0
    \\
     \vdots &    \vdots     &            \vdots                     &   \vdots
        &    \ddots   &   \vdots      &         \vdots
    \\
      0    &    0     &        0      &           0
        &   \cdots    &    (d_{n-1} + c_{n-1} ) &  -c_{n-1}
    \\
      0    &    0     &        0      &           0
        &   \cdots    &  -d_{n}   &      (d_{n} + c_{n} )
  \end{matrix} \right|
\end{equation}
then the determinant is given by,
\[
D_n
 =    ( d_{1} \cdots d_{n} )
    + ( d_{1} \cdots d_{n-1} c_{n} )
    + \cdots
    + ( d_{1} c_{2} \cdots c_{n-1} c_{n} )
    + ( c_{1} \cdots c_{n} ).
\]
\end{Proposition}
\begin{proof}
We proceed by induction and use the general form for determinants of tri-diagonal matrices.
The formula is easily derived from the definition of the determinant and has the form,
\begin{equation}
\label{prop_trimat_det!eq_Dn_iteratte}
D_{n}(1,n) =  (d_{1}  + d_{1}) D_{n-1}(2,n) - c_{1} d_{2} D_{n-2}(3,n),
\end{equation}
where $D_{n-1}(k_1,k_2)$ is the determinant of the $k_2 - k_1 + 1$ square matrix
with row and column indices between $k_1$ and $k_2$ inclusive.
The cases $n =3$ and $n = 4$ are easily computed directly.
The case for general $n$ can be proved
by induction using equation \ERef{prop_trimat_det!eq_Dn_iteratte}.
\end{proof}

\section{Characteristic Polynomial Near 0}
\label{sec!charpolyzero}

The characteristic equation of the system is given in Corollary \ref{nnsystem!cor_eigsofpoly}.
The Taylor expansion of this equation in the variable $\phi$,
described in Section \ref{charpolyexpand!start},
results in an approximation for the roots that lie near the double root at $\nu = 0$.
The main result is given in Theorem \ref{cpolyzero!thm_secorder}.
See Section \ref{charpolyexpand!start} for a detailed derivation
of some of the coefficients needed in the expansion.

\begin{Theorem}
\label{cpolyzero!thm_secorder}
With the assumptions described at the start of Section \ref{sec!symlap}
the characterisitic polynomial $P_{\phi}( \nu )$ has a double zero when $\phi = 0$.
The coefficients $c_1, c_2$, in the expansion of equation \ERef{charpolyexpand!eq_gamma_taylor}
are given by,
\begin{align}
\label{charpolyzero!eq_c1}
  \gamma'(0) &= \pm \frac{i}{ p } \sqrt{ \frac{ 1 }{ 2 \Avg \left( 1/g_{x}^{(k)} \right) } }
  \\
\label{charpolyzero!eq_c2}
  \gamma''(0) &= - \frac{ \Avg \left( \frac{g_v^{(k)}}{ g_x^{(k)} g_x^{(k)} } \right) }
                { 2 \left( p \Avg \left( \frac{1}{ g_x^{(k)} } \right) \right)^{2} }
\end{align}
This means that the zero locus, near the double root, is approximated
by (see equation \ERef{charpolyexpand!eq_gamma_taylor}),
\begin{equation}
\label{cpolyzero!eq_zerolocus_secorder}
\gamma( \phi )
  = \pm i \sqrt{ \frac{ 1 }{ 2 \Avg \left( 1/g_{x}^{(k)} \right) } } \left( \frac{\phi}{p} \right)
  - \frac{  \Avg \left( \frac{g_v^{(k)}}{ g_x^{(k)} g_x^{(k)} } \right) }
         { \left( 2 \Avg \left( 1/g_{x}^{(k)} \right) \right)^{2} }
        \left( \frac{\phi}{p} \right)^2
  + \cdots,
\end{equation}
\end{Theorem}
\begin{proof}
To find $\gamma'(0)$ we start with Proposition \ref{agentseq!prop_gamma_expand}.
Then use Proposition \ref{agentseq!prop_term_a20} equation \ERef{agentseq!eqnu2_second}
and Proposition \ref{agentseq!cor_deriv_akm} equation \ERef{agentseq!cor_eq_a_02}.
We get,
\begin{align*}
\gamma'(0) &= \pm \sqrt{ - \frac{ a_{02} }{  a_{20} } }
     = \pm \sqrt{ - \frac{ 1 }{ (2p) \sum\limits_{k=0}^{p-1} {1/g_{x}^{(k)} } } }
\end{align*}
Equation \ERef{charpolyzero!eq_c1} follows.

To find $\gamma''(0)$ use Propositions \ref{agentseq!prop_gamma_expand} and \ref{agentseq!prop_gamma_expand}.
\begin{align*}
\gamma''(0)  &= - \frac{ a_{30} (\gamma'(0))^3 + a_{21} (\gamma'(0))^2 + a_{12} \gamma'(0) + a_{03} }
               { a_{20} \gamma'(0) }
    = - \frac{ a_{30} (\gamma'(0))^2 + a_{12} }
             { a_{20} }
    = \frac{ a_{30} a_{02} - a_{12} a_{20} }{ a_{20}^{2} }
\end{align*}
The coefficients $a_{20}$ and $a_{30}$ are computed in Proposition \ref{agentseq!prop_term_a30}.
The coefficients $a_{12}$ and $a_{02}$ are computed in Proposition \ref{agentseq!cor_deriv_akm}.
Assemble the coefficients, cancel the terms $\prod_{j=0}^{p-1} { g_{x}^{(j)} }$ and $2^{p}$
and we get
\begin{align*}
\gamma''(0)  &= \frac{ a_{30} a_{02} - a_{12} a_{20} }{ a_{20}^{2} }
  \\
  &= \frac{ 1 }{  4 p^4 \Avg \left( 1/g_{x}^{(k)} \right)^{2} }
     \left(  (2p)
    \left(
      \sum\limits_{k=0}^{p-1} {
        \frac{1}{g_x^{(k)}}
      }
      \sum\limits_{j=0}^{p-1} {
        \frac{g_v^{(j)}}{ g_x^{(j)}}
      }
     -
      \sum\limits_{k=0}^{p-1} { \frac{g_v^{(k)}}{ g_x^{(k)} g_x^{(k)} } }
    \right)
    -
    \left(
      (2p)
      \sum\limits_{j=0}^{p-1} { \frac{ g_{v}^{(j)} }{ g_{x}^{(j)} } }
      \sum\limits_{k=0}^{p-1} { \frac{ 1 }{  g_{x}^{(k)} } }
    \right)
  \right)
   \\
  &= - \frac{ (2p) \sum\limits_{k=0}^{p-1} { \frac{g_v^{(k)}}{ g_x^{(k)} g_x^{(k)} } } }
          {  4 p^4 \Avg \left( 1/g_{x}^{(k)} \right)^{2} }
   = - \frac{ \Avg \left( \frac{g_v^{(k)}}{ g_x^{(k)} g_x^{(k)} } \right) }
        { 2 p^2 \Avg \left( 1/g_{x}^{(k)} \right)^{2} }
\end{align*}
\end{proof}

\begin{Remark}
Notice that the derivative $\gamma'(0)$ is imaginary and the second order $\gamma''(0)$ is real and negative.
This is consistent with the stability of Theorem \ref{symlap!thm_LvLx_stable}.
\end{Remark}

We define two positive values $c_{1}, c_{2}$ by,
\begin{align}
\label{cpolyzero!eq_def_c1}
  c_{1} = p |\gamma'(0)|
              = \sqrt{ \frac{ 1 }{ 2 \Avg \left( 1/g_{x}^{(k)} \right) } }
  \\
\label{cpolyzero!eq_def_c2}
  c_{2} =  p^2 | \gamma''(0) |
              = \frac{ \Avg \left( \frac{g_v^{(k)}}{ g_x^{(k)} g_x^{(k)} } \right) }
                     { \left( 2 \Avg \left( \frac{1}{ g_x^{(k)} } \right) \right)^{2} }
\end{align}

Recall that $\phi = \frac{2 \pi }{q} m$.
We re-write equation \ERef{cpolyzero!eq_zerolocus_secorder} as,
\begin{align}
\gamma_{m}
  &= \pm i c_{1} \left( \frac{2 \pi }{p q} m \right)
     - c_{2} \left( \frac{ 2 \pi }{p q} m \right)^2
     + \cdots
\label{cpolyzero!eq_gamma_hats_N}
   = \pm i c_{1} \left( \frac{2 \pi }{N} m \right)
     - c_{2} \left( \frac{ 2 \pi }{N} m \right)^2
     + \cdots
\end{align}

Recall that we shuffled coordinates to get the eigenvector in equation \ref{linnn!eq_eigenvect}.
This eigenvector, in the original coordinates, is
\[
 \begin{pmatrix}
       \Vect{ \epsilon } \otimes \Vect{v}_{m}(\phi)        \\
       \nu  \Vect{ \epsilon } \otimes \Vect{v}_{m}(\phi)
  \end{pmatrix}
\]
The solutions have time dependence through the factor,
\[
\exp\left( \gamma(\phi) t  \right)
                \exp\left( \frac{ 2 \pi i}{q} \mu m \right)
  = \exp\left( \frac{2 \pi i }{pq } m \left( \pm c_1 t + p \mu \right) \right)
      \exp\left( - c_2 \left( \frac{ 2 \pi }{ p q } m \right)^{2} t  \right)
\]

The signal velocity is determined by the factor $\pm c_1 t + p \mu$.
There are two waves going in opposite directions with equal velocities.
We see that the time for the signal to travel $p$ agents is given by $c_{1} T_{p} = p$ so that,
\begin{equation}
\label{soln!eq_T1_estimate}
T_{1}
  = \frac{p}{c_{1} }
  = p \sqrt{2 \Avg \left( 1/g_{x}^{(k)} \right)}.
\end{equation}
The factor, $\exp\left( - c_2 \left( \frac{ 2 \pi }{ p q } m \right)^{2} t  \right)$
is a damping term that introduces a dispersion relation.
The larger $c_2$ is the larger the damping factor.

In our original problem we assumed the $N = pq$ agents were assembled
as $q$ groups of $p$ randomly weighted agents.
The values $c_{1}$ and $c_{2}$ are averages over the $p$ agents
which equals the average over the $N$ agents.
What if we have $N$ randomly weighted agents.
This amounts to setting $q = 1$.
Does equation \ERef{cpolyzero!eq_gamma_hats_N} still apply?
In the following section we shall run a variety of simulations
on $N$ uniquely weighted agents and compare to our solutions.

\section{Simulations}
\label{sec!simul}

The computation in sections \ref{charpolyexpand!start} and \ref{sec!charpolyzero}
assumed a condition that is a bit stronger than periodic boundary conditions.
It assumed that $p$ independently weighted agents were duplicated $q$ times
for a total sequence of $N = pq$ agents and that this $N$-string satisfies periodic boundary conditions.
In this simulation section we shall forgo this assumption and
shall set $q = 1$ and set the sequence on the real line without periodic boundary conditions.
The scales, $g_{x}^{(\alpha)}, g_{v}^{(\alpha)}$, are chosen independently.
We will demonstrate that the periodicity of the $p$-sequence is not required in practice.
In these simulations we shall assume the conditions
in equations \ERef{symlap!assume_rhox_equal} and \ERef{symlap!assume_rhov_equal}
so that $\rho_{x,1} = \rho_{x,-1}$ and $\rho_{v,1} = \rho_{v,-1}$.

\begin{figure}[h!]
\centering
\begin{subfigure}{.45\textwidth}
  \centering
  \includegraphics[width=\textwidth]{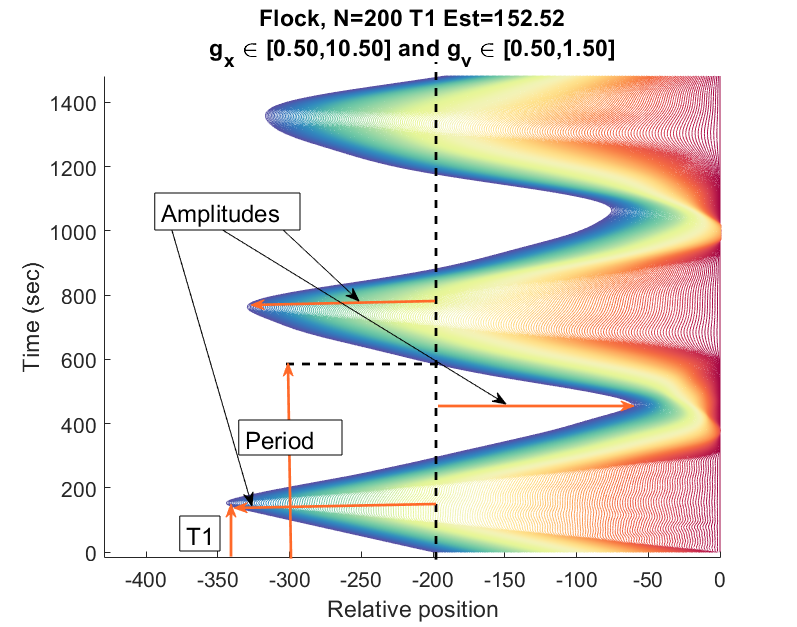}
  \caption{Flock with times indicated}
  \label{fig:sim_flock_Init_N200}
\end{subfigure}			%
\begin{subfigure}{.45\textwidth}
  \centering
  \includegraphics[width=\textwidth]{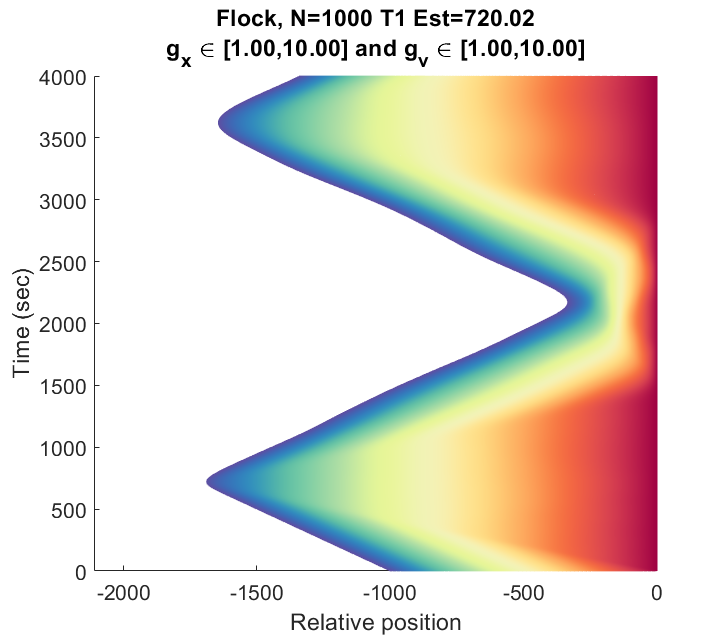}
  \caption{Flock of 1000 agents}
  \label{fig:sim_flock_N1000}
\end{subfigure}
  \caption{Flock Examples}
  \label{fig:sim_basicplot}
\end{figure}

\begin{figure}[h!]
\centering
\begin{subfigure}{.45\textwidth}
  \centering
  \includegraphics[width=\textwidth]{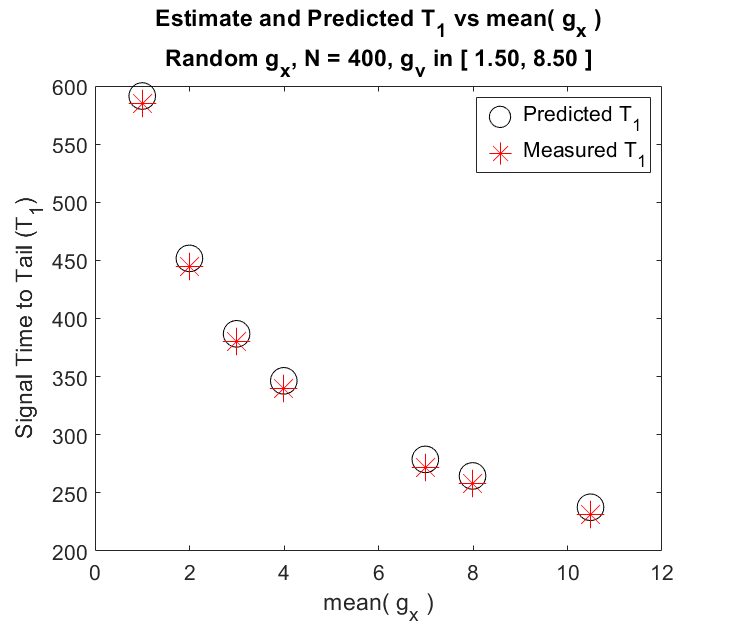}
  \caption{Estimate for $T_1$ when $g_x$ is random from uniform distribution}
  \label{fig:sim_T1vsGx_Random}
\end{subfigure}
\begin{subfigure}{.45\textwidth}
  \centering
  \includegraphics[width=\textwidth]{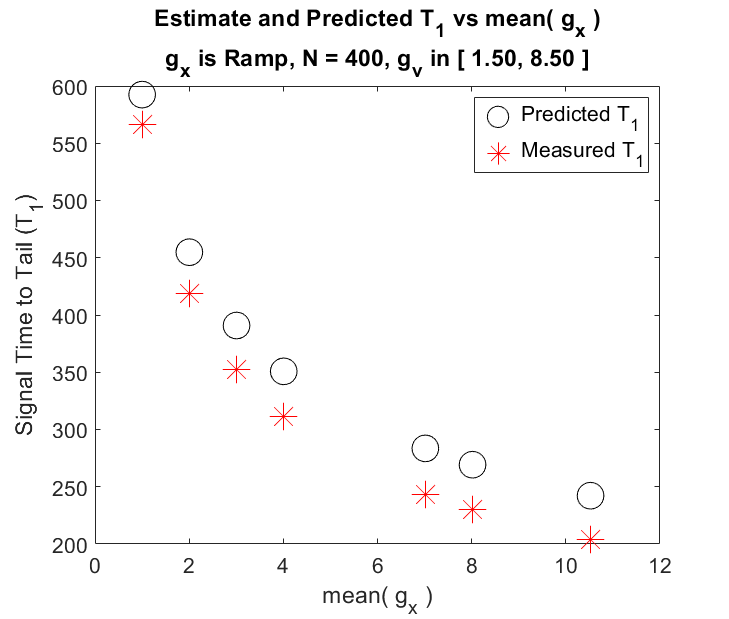}
  \caption{Estimate for $T_1$ when $g_x$ is a ramp}
  \label{fig:sim_T1vsGx_Ramp}
\end{subfigure}
  \caption{Prediction of the first response time $T_1$ compared to computational predictions}
  \label{fig:sim_basicplot2}
\end{figure}

In Figure \ref{fig:sim_flock_Init_N200} we show a linear flock with $N = p = 200$ agents.
The plot shows the simulation estimate for $T_1$, which we derived in equation \ERef{soln!eq_T1_estimate}.
The signal velocity actually depends on the constituent waves in the wave packet
that moving from the leader to the tail.
The simple model used to estimate $T_1$ in the simulations is to pick the time
where the distance to the leader is the greatest.
Also shown in the plot are the distance used to estimate the period
and the peaks used to evaluate amplitude ratios.
We estimate the amplitude by choosing the point with the largest distance to the leader,
just as we did to estimate $T_1$.
The values $g_{x}^{(\alpha)}$ are independently chosen from a uniform distribution on $[ 0.50, 10.50 ]$
and the $g_{v}^{(\alpha)}$ from $[ 0.50, 1.50 ]$.
For this particular snapshot of random variables the value $T_1$ in equation \ERef{soln!eq_T1_estimate}
gives $T_1 = 153.0$ which is extremely close to the observed value.
In this graph each agent trajectory is in a different color so that flock dynamics are evident.

In Figure \ref{fig:sim_flock_N1000} we have a basic simulation of $N=1000$.
The weights for $g_{x}^{(\alpha)}$ are independently chosen from a uniform distribution on $[1, 12]$
and $g_{v}^{(\alpha)}$ from a uniform distribution on  $[ 1, 12 ]$.
In this simulation the estimated $T_1 = 720$ and the computed value is $T_1 = 721$,
which is, again, quite close.

Equation \ERef{soln!eq_T1_estimate} gives a prediction for the first response time.
In figure \ref{fig:sim_T1vsGx_Random} we compare the computed value of $T_1$
to the simulation estimate, as we vary the mean of $g_{x}^{(\alpha)}$.
For each test we select $N = p$ values $g_{x}^{(\alpha)}$ independently
from a uniform distribution with a mean shown in the x-axis.
In figure \ref{fig:sim_T1vsGx_Ramp} we run simulations where $g_{x}^{(\alpha)}$
form a ramp as $\alpha = 0, \cdots p-1$.
The ramp starts with $g_{x}^{(0)}$ small and increases so that $g_{x}^{(p-1)}$ has a maximum value
according to the formula,
\[
 g_{x}^{(\alpha)} = g_{low} + \alpha \frac{ ( g_{hi} - g_{low} ) }{ p - 1 },
   \;\;
   \alpha = 0, 1, \cdots, p-1.
\]
The $T_1$ estimates for the random distribution are very close to the predicted value.
The ramp ``distribution'' is not random and it has a large discontinuity at $\alpha = 0$.
Although these are clear differences from the uniform distribution, we do not yet understand why the ramp distribution makes the $T_1$ prediction less accurate.

\begin{figure}[h!]
\centering
\begin{subfigure}{.45\textwidth}
  \centering
  \includegraphics[width=\textwidth]{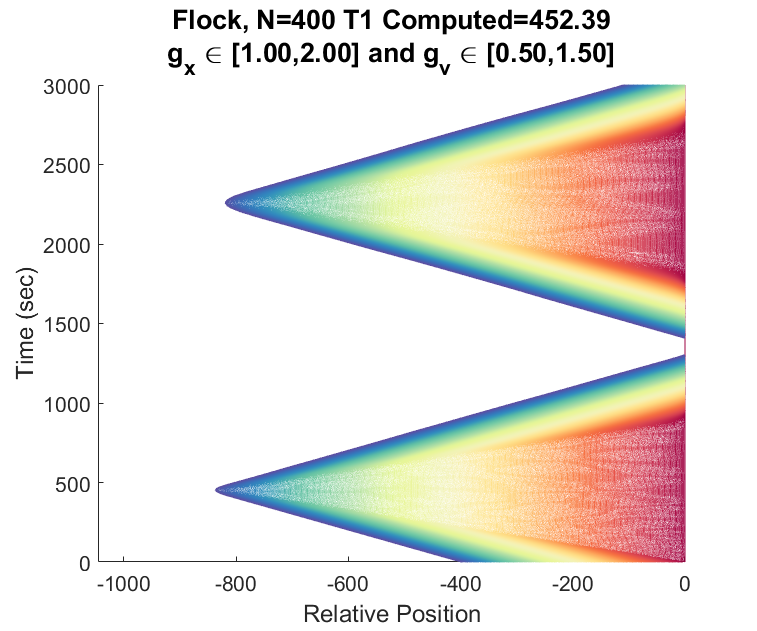}
  \caption{Mean $g_v$ is smaller so $c_2$ is smaller}
  \label{fig:sim_rhoV_Low_N400}
\end{subfigure}			%
\begin{subfigure}{.45\textwidth}
  \centering
  \includegraphics[width=\textwidth]{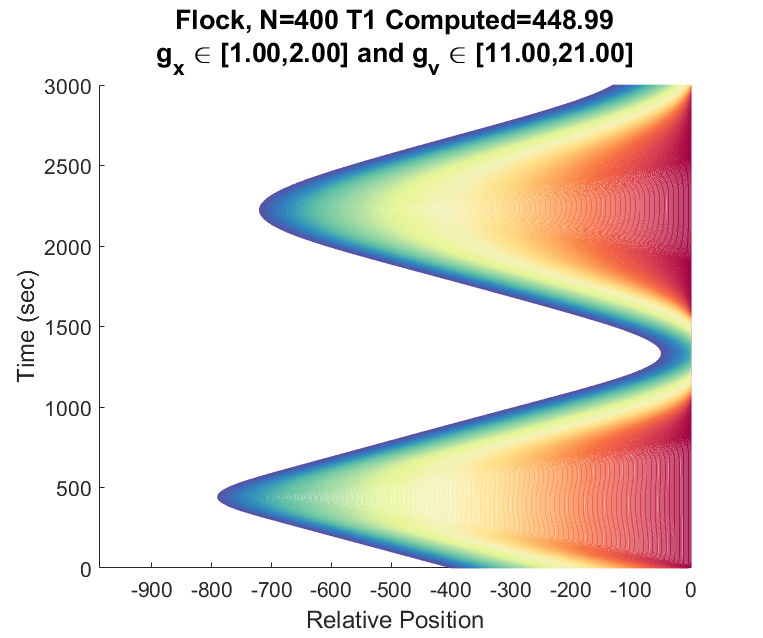}
  \caption{Mean $g_v$ is larger so $c_2$ is larger}
  \label{fig:sim_rhoV_High_N400}
\end{subfigure}
  \caption{Changing $c_2$ changes dispersion relation}
  \label{fig:sim_unstable_N1000a}
\end{figure}

The value $c_1$ in equation \ERef{cpolyzero!eq_def_c1} determines the first response
as it appears in an imaginary expential.
The value $c_2$ in equation \ERef{cpolyzero!eq_def_c2} appears in a real exponential
and is negative which indicates a stable solution.
This is a dispersion term and tends to disperse the pure frequencies
so that the amplitude edges are muted.
Figure \ref{fig:sim_rhoV_Low_N400} shows a plot with small $c_2$
and figure \ref{fig:sim_rhoV_High_N400} is the same system
except the mean value of $g_{v}^{(\alpha)}$ is larger.
This will increase the value of $c_2$ and the resulting peaks are rounder.

\begin{figure}[h!]
\centering
\begin{subfigure}{.45\textwidth}
  \centering
  \includegraphics[width=\textwidth]{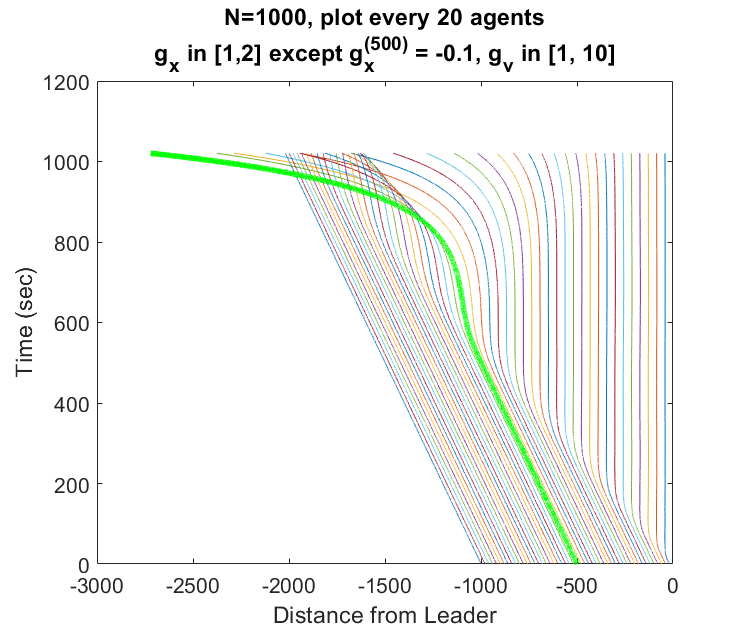}
  \caption{One agent (in green) has negative weight}
  \label{fig:sim_unstable_gxneg_N1000}
\end{subfigure}			%
\begin{subfigure}{.45\textwidth}
  \centering
  \includegraphics[width=\textwidth]{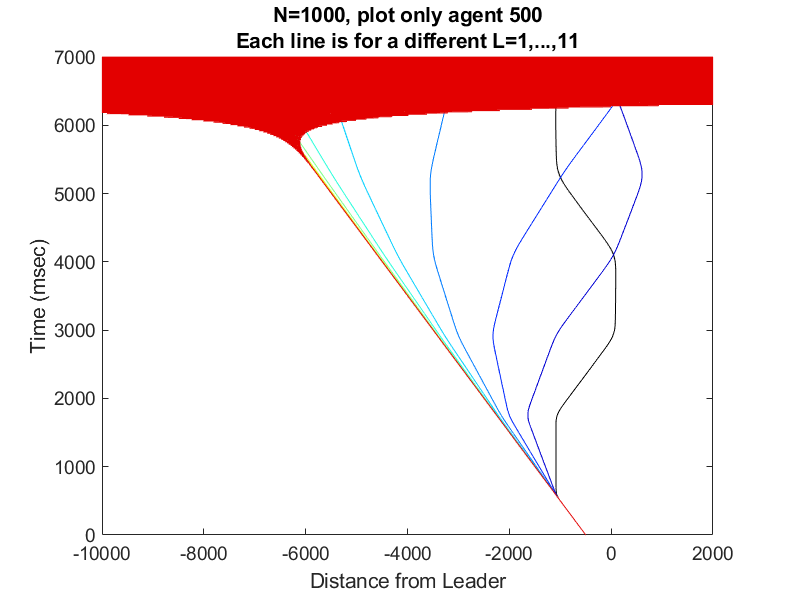}
  \caption{Set $\rho_{x,2}^{(j)}=0.25$ for agents $j=500, \cdots, 500+(L-1)$
           where $L = 0, \cdots 11$.
           The more rows with $\rho_{x,1} = 0.25$ the faster the onset of instability.}
  \label{fig:sim_unstable_rhoxvar_N1000}
\end{subfigure}
  \caption{Simulations demonstrating instability}
  \label{fig:sim_unstable_N1000}
\end{figure}

The numerics indicate that
the system is stable only when $g_{x}^{(\alpha)}$ are non-negative for all $\alpha$.
Setting $g_{x}^{(\alpha)}$ negative, for a single agent has an immediate effect on system stability.
Figure \ref{fig:sim_unstable_gxneg_N1000} shows a simulation
where a single agent has negative weight.
The plot shows a subset of the flock.
The agent with negative weight is shown in green.
When the signal to move reaches this agent, the agent moves in the opposite direction
and the system starts on an unstable trajectory.
If the weight of this single agent was zero
then the agents to the left of it would never react to the motion of the leading agent
since we are only including nearest neighbor interactions.

Theorem \ref{nnsystem!thm_nec_for_stab} states a necessary condition for stability.
The difficulty with the theorem is that
both $\prod_i\;\rho_{x,1}^{(i)}$ and $\prod_i\;\rho_{x,-1}^{(i)}$
are very small and go to zero as $N \to \infty$.
If we have $\rho_{x,1} = \rho_{x,-1}$ then the condition is satisfied but
changing a single value of $\rho_{x,1}^{(\alpha_0)}$ to $-0.25$ should result in an unstable system.
This is difficult to show numerically.
In Figure \ref{fig:sim_unstable_N1000}, we plot the middle agent for a series of simulations.
Each simulation has $\rho_{x,1} = \rho_{x,-1} = -1/2$
except for $L$ agents centered around agent $500$.
For the given simulations, the system is unstable when $L \ge 5$.
The greater $L$ the faster the onset of the instability.

\begin{figure}[h!]
\centering
\begin{subfigure}{.45\textwidth}
  \centering
  \includegraphics[width=\textwidth]{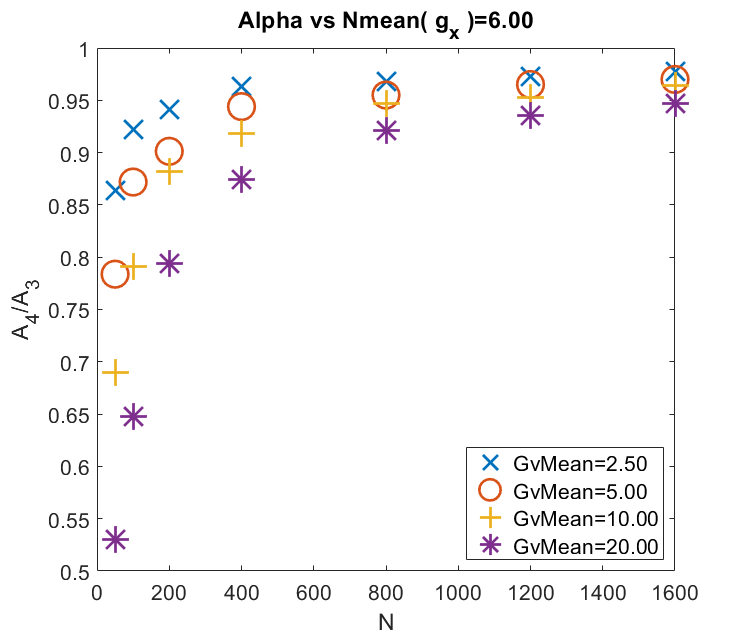}
  \caption{Plot showing $\alpha$ as $N \to \infty$}
  \label{fig:sim_AlphaVsN_N1000_Super}
\end{subfigure}			%
\begin{subfigure}{.45\textwidth}
  \centering
  \includegraphics[width=\textwidth]{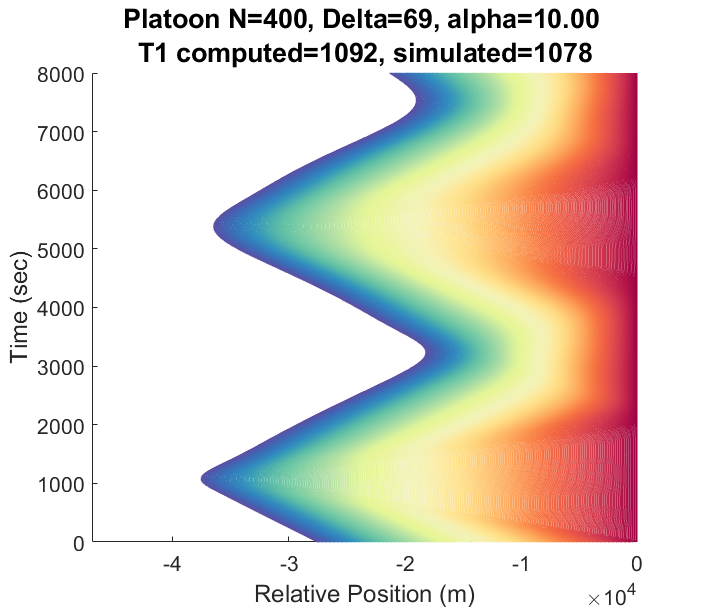}
  \caption{Simulation of truck convoy with 400 agents.}
  \label{fig:sim_TruckConvoy_N400}
\end{subfigure}
\end{figure}

A previous work \cite{Herbrych2015DynamicsOL} derives a condition
on the ratio of successive flock amplitudes.
There it is shown that as $N \to \infty$ then the ratio goes to $1.0$.
In Figure \ref{fig:sim_AlphaVsN_N1000_Super} we plot amplitudes $A_4/A_3$
for various distributions of $g_v$.
We do not use the ratio $A_{2} / A_{1}$ since the first amplitude is close
to the initial condition and behaves differently than subsequent peaks.
The curves depend on the mean of $g_{v}^{(\alpha)}$
as altering $g_{v}^{(\alpha)}$ alters $c_2$ (see equation \ERef{cpolyzero!eq_def_c2}).
We cannot conclusively conclude that this is exactly correct, and more simulation work is required.

We conclude with the realistic system discussed in the introduction.
With this example we demonstrate that the tools presented in this paper
can be used to analyze more complicated and realistic problems.
We make some rough estimates in this next section.
A automotive engineer could refine these numbers with more realistic estimates.
We model a convoy of $N$ trucks traveling on the highway.
The convey attempts to keep a fixed spacing between trucks
and the trucks are all different.
As the convoy travels, lighter cars might, inadvertently, enter the convoy
creating a 1-dimensional convoy with very different agents.
To use our model, we must estimate the agent weights $g_{x}^{k}$ and $g_{v}^{k}$.
The weights for agent $k$ are force coupling divided by the mass of the agent.
The mass of a 18 wheel truck is somewhere between 14 and 40 thousand kilograms
and the coupling force is determined by the torque of the engine.
To make things simpler we shall assume the force divided by the mass
produces a given acceleration and we can estimate this acceleration.
For example, a truck can accelerate from $0$ to $60$ mph = $26.8$ m/sec in 1 to 5 minutes.
So we take our truck weights $g_{x}^{k}$ in the range,
\[
 g_{x}^{k} \in [ 26.8 / 60, 26.8 / 300 ],
   \; \;
   k \RText{ is a truck}.
\]
We insert cars into the convoy by randomly replacing $10\%$ of the agents with lighter cars.
Cars, typically, have higher power to mass and so have larger weights.
We take a collection of cars that accelerate from $0$ to $60$ in a range of $6$ to $20$ seconds,
so that for car agents,
\[
 g_{x}^{k} \in [ 26.8 / 6, 26.8 / 20 ],
   \; \;
   k \RText{ is a car}.
\]
To guarantee stability we take $G_v = \alpha G_x$ where $\alpha = 10.0$.
Increasing $\alpha$, as we've seen, will increase the damping.
U.S. 18-wheel trucks are typically around $23$ meters long.
The convoy attempts to keep a distance of $3 \times 23 = 69$ meters between the agents.

The simulation results are shown in figure \ref{fig:sim_TruckConvoy_N400}.
The convey of $400$ starts with a spacing of $69$ meters so is almost $28km$ long.
The first track suddenly increases its speed $10$ meters/second
and it takes $1095$ seconds for the signal to reach the last truck.
The time duration is long because the weights are small (e.g the trucks accelerate slowly).
The distance from the leader to the tail lengthens to $37.5km = 94$m/truck
before the tail starts to catch up.
The damping is not critical, and the tail overshoots the optimal distance
and the entire convey shrinks to $18km = 45$ m/truck before expanding again.
This simulation assumed $\rho_{v,1} = \rho_{v,-1}$.
There is some evidence that system is more responsive if this condition is removed.
We will explore that in a subsequent paper.


\newpage

\bibliographystyle{plain}
\bibliography{./SymmLap-DistinctGs-Refs}


\end{document}